\documentclass[11pt,a4paper]{article} 
\usepackage[utf8]{inputenc} 
\usepackage[T1]{fontenc} 
\usepackage{amsmath,amssymb,amsthm,amsfonts} 
\usepackage{mathtools} 
\usepackage{graphicx} 
\usepackage{xcolor} 
\usepackage{hyperref} 
\hypersetup{colorlinks=true, linkcolor=blue, citecolor=blue, urlcolor=blue}
\usepackage[nameinlink]{cleveref}
\usepackage{enumitem}
\usepackage{geometry}
\usepackage{algorithm}
\usepackage{algpseudocode}
\theoremstyle{plain} 
\newtheorem{theorem}{Theorem}[section] 
\newtheorem{proposition}[theorem]{Proposition} 
\newtheorem{lemma}[theorem]{Lemma} 
\newtheorem{corollary}[theorem]{Corollary} 
\theoremstyle{definition} 
\newtheorem{definition}[theorem]{Definition} 
\newtheorem{example}[theorem]{Example} 
\theoremstyle{remark} 
\newtheorem{remark}[theorem]{Remark}

\title{On the Number of Iterations in a Kruskal-Type Process} 
\author{Samuel G. Corregidor\footnote{Dpto. Tecnologías de la Información, Escuela Politécnica Superior, Universidad CEU San Pablo, Boadilla del Monte, 28668 Madrid. Email: \texttt{samuel.gutierrezcorregidor@ceu.es}}} 
\date{} 
\begin{document}

\maketitle 

\begin{abstract}
    Motivated by the Kruskal counting procedure, this paper investigates a combinatorial elimination process defined on a deck consisting of $m$ copies of each value $1,2,\ldots,n$. Starting from an arbitrary arrangement of the $mn$ cards, the process repeatedly removes a number of consecutive cards equal to the value of the first remaining card. This generates a sequence of elimination steps whose length determines $T$, representing the total number of iterations required to exhaust the deck. The main objective of this paper is to analyze an upper bound $\sigma(m,n)$ for $T$, valid for every initial configuration of the deck, and to show that this bound is attained by a specific configuration.
\end{abstract} 


\section{Introduction}

In \cite{Lagarias2009}, a problem devised by Martin D. Kruskal is addressed. The problem can be formulated as follows: start with a shuffled deck. A spectator secretly chooses an initial number $k$. Then, $k$ cards are counted, the card reached becomes a key card, and the value of that card becomes the new number of steps. This process is repeated until the end of the deck is reached.

The so-called \emph{Kruskal Count} problem is a paradigmatic example of the convergence of trajectories in discrete random processes. Although it was originally popularized as a card trick by Martin D. Kruskal, subsequent studies have shown that the underlying mechanism can be rigorously interpreted by \emph{coupling} techniques in Markov chains \cite{Lagarias2001,Lagarias2009,Lindvall2002}.

The main feature of the process is that trajectories starting from different states have a high probability of meeting after a finite number of steps and then evolving identically afterwards. This phenomenon of probabilistic synchronization occurs naturally in various contexts of information theory \cite{CoverThomas2006,Dobrushin1967,IyengarSiegelWolf2011}, theoretical computer science \cite{BerstelPerrinReutenauer2010}, including the synchronization of variable-length codes \cite{GilbertMoore1959,Higgie1993,Titchener1997}, mechanisms of resynchronization in encoded sequences \cite{Banerjee2023,Dobrushin1967}, graph-based software tamper protection mechanisms \cite{Dedic2007}, synchronization of finite automata, directed graph structures \cite{Trahtman2008,Trahtman2009} and cryptographic algorithms based on pseudo-random walks \cite{CharlesGorenLauter2008,JaoMillerVenkatesan2009}.

In particular, a close conceptual analogy has been pointed out with Pollard's kangaroo method for the calculation of discrete logarithms, where the detection of collisions between random trajectories constitutes the central element for the efficient computation of the solution \cite{Pollard1978}. Likewise, the study of the number of steps required to complete these procedures is closely related to the theory of stochastic processes and \emph{stopping times} \cite{Cinlar2013,Durrett2019,Ross1996}.

In this paper, a combinatorial problem defined on a deck composed of $m$ copies of each of the numbers $1,\ldots, n$ is studied, so that the total number of the cards is $mn$. Given an initial arrangement of the cards, the process evolves through an iterative procedure of elimination. At each step, the value of the first remaining card is recorded, and that many consecutive cards are removed. Since the number of cards is finite, the process is repeated on the resulting deck until all cards have been removed.

In order to analyze this dynamic, for each initial state of the deck, which will be represented by an ordered sequence of the value of cards $\mathcal{C} =(C_1,\ldots,C_L)$, we associate another sequence $(X_t)$ of the values that intervene in the successive stages of elimination (see Definition \ref{d:sucesion}).

The main objective of this paper is to study the number of iterations $T$ needed to complete the entire process. In section 3, an initial state (\emph{maximal}) is defined for different values of $n$ (triangular or not triangular). A function $\sigma(m,n)$ is defined for these initial states to compute $T = \sigma(m,n)$ (see Corollary \ref{c:cota}). Indeed, for every initial state, this function is an upper bound, resulting from Theorem \ref{t:cotas} that $m\leq T \leq \sigma (m,n)$. 

\section{Preliminaries} 

\begin{definition}\label{d:sucesion}
    Let $L\in\mathbb{N}$ and $\mathcal{C} = (C_i)_{i=1} ^L$ be a finite sequence of $L$ positive integers, for which $C_k = 0$ for every $k > L$. The sequence $(X_t)$ is defined recursively as follows: 
    
    \[ \begin{cases}
        X_0 = 1 \\
        X_{t+1} = X_t + C_{X_t}, \qquad \text{for every } t\in\mathbb{N}
    \end{cases} \] 

    It is denoted: 
    
    \[T(\mathcal{C}) = \min\{ t\in\mathbb{N} : X_t > L \}\]
    
    For convenience, the numbers in the sequence $\mathcal{C}$ will be referred to as \emph{values}.
\end{definition}

\begin{example}\label{e:estados} 
    Suppose that the \emph{values} are given by a linear expression $C_k = ak+b $ for $a,b>0 $. In that case, by Definition \ref{d:sucesion}: 
    
    \[
        X_{t+1}=(a+1)X_t+b
    \]
    
    The above equation is a first-order non-homogeneous linear difference equation. Using the fixed point method (\cite{elaydi2005difference,kelley2001difference}), the equilibrium is obtained as $ X^\ast=-\frac{b}{a}$. If we define $Y_t=X_t-X^\ast$, then $Y_{t+1}=(a+1)Y_t$. Therefore, $Y_t=(a+1)^tY_0$ and consequently,
    
    \[ X_t=\frac{a+b}{a}(a+1)^t-\frac{b}{a}, \qquad \text{ for every }t \geq 0 \]
    
    To determine $T(\mathcal{C})$, we seek the smallest $t\in\mathbb{N}$ such that $ X_t>L$. By the above equation we obtain:
    
    \[ \frac{a+b}{a}(a+1)^t-\frac{b}{a}>L \] 
    
    An equivalent inequality is
    
    \[ (a+1)^t>\frac{aL+b}{a+b} \] 
    
    Thus,
    
    \[ T(\mathcal C) = \left\lfloor \log_{a+1} \left(\frac{aL+b}{a+b}\right) \right\rfloor + 1 \]
\end{example}

\begin{lemma}\label{l:sucesion}
    For every finite sequence $\mathcal{C} = (C_i)_{i=1} ^L$ of positive integer \emph{values}: 
    \[ 
        X_t = 1 + \sum _{j=0} ^{t-1} C_{X_j},\qquad \text{for every } t\geq 1.
    \] 
    In addition, $T(\mathcal{C}) = | \{ t\in\mathbb{N} : C_{X_t} > 0 \} |$.
\end{lemma}

\begin{proof}
    The proof proceeds by induction on $t$. For the base case $t = 1$, we have $X_1 = X_0 + C_{X_0} = 1 + C_{X_0}$. Assume that the result holds for $t$. Then,
    
    \[
        X_{t+1} = X_t + C_{X_t} = 1 + \sum_{j=0} ^{t-1} C_{X_j} + C_{X_t} = 1 +  \sum_{j=0} ^{t} C_{X_j}
    \] 
    
   In addition, if $0\leq t < T(\mathcal{C})$, then $X_t \leq L$ and consequently, $C_{X_t} > 0 $. Therefore, $\{ 0, \ldots T(\mathcal{C}) - 1 \} \subseteq \{ t\in\mathbb{N} : C_{X_t} > 0 \}$.

   Conversely, if $C_{X_t} > 0$, then $X_t \leq L$ and consequently, $0\leq t < T(\mathcal{C})$. Finally, we obtain 
   
   \[
        \{ 0, \ldots T(\mathcal{C}) - 1 \} = \{ t\in\mathbb{N} : C_{X_t} > 0 \}
    \]
   
   Since $|\{ 0, \ldots T(\mathcal{C}) - 1 \}| = T(\mathcal{C})$, then $T(\mathcal{C}) = | \{ t\in\mathbb{N} : C_{X_t} > 0 \} |$.
\end{proof}

\begin{lemma}\label{l:sucesion_dos} For every finite sequence $\mathcal{C} = (C_i)_{i=1} ^L$ of positive integer \emph{values}: 

\[ 
    X_t = X_{t'} + \sum _{j=t'} ^{t-1} C_{X_j}, \qquad \text{for every } t>t'\geq 0.
\]
\end{lemma}
\begin{proof}
    Applying Lemma \ref{l:sucesion} twice yields: 
    
    \[
        X_t = 1 + \sum _{j=0} ^{t-1} C_{X_j}, \qquad X_{t'} = 1 + \sum _{j=0} ^{t'-1} C_{X_j} 
    \] 
        
    Therefore, subtracting the two expressions, we obtain
    
    \[
        X_t - X_{t'} = \sum _{j=t'} ^{t-1} C_{X_j}
    \]
\end{proof}

\begin{remark}
    Lemma \ref{l:sucesion} ensures that $T(\mathcal{C})$ counts the number of positive \emph{values} appearing in the sequence $(X_t)_{t = 0} ^\infty$.
\end{remark}

In this paper, we study the sequence $\mathcal{C}$ under specific restrictions on its \emph{values}. To this end, we introduce the following definition.

\begin{definition}\label{d:inicial}
    Let $m,n\in\mathbb{N}$ be fixed. We will say that $\mathcal{C}_n^m = (C_i)_{i=1} ^L $ is an initial state whenever the following conditions hold: 
    \begin{enumerate}
        \item $L = mn$.
        \item $C_i \in \{1,\ldots , n \}$ for every $i = 1, \ldots , mn$.
        \item $| \{ i : C_i = j \} | = m$ for every $j = 1, \ldots , n$.
    \end{enumerate}
\end{definition}

\begin{example} 
    An initial state is simply a permutation of the multi set
    
    \[
        M = \{ 1^{(m)} , \ldots , n^{(m)} \}
    \]
    
    where each element $i^{(m)}$ has multiplicity $m$. For example, when $m=n=3$, a possible initial state is
    
    \[
        \mathcal{C}_3 ^3 = (1,2,3,1,2,3,1,2,3)
    \]
    
    In this case, $T(\mathcal{C}_3 ^3) = 6$.
\end{example}

\section{Main Results}

A natural number $n$ is called \emph{triangular} if $8n+1$ is an odd perfect square. Equivalently, there exists $k\in\mathbb N$ such that $8n+1=(2k+1)^2$.

\begin{remark}
    The above condition is equivalent to the classical representation
    
    \[
        n=\frac{k(k+1)}2
    \]
    
    for some $k\in\mathbb N$. 
\end{remark}

\begin{definition}\label{d:maximal_triangular}
    Let $m,n,k\in\mathbb{N}$. An initial state $\hat{\mathcal{C}}_n ^m = (C_i) _{i=1} ^L$ is said to be maximal triangular whenever the following conditions hold:
    \begin{enumerate}
        \item $L = mn$.
        \item $8n+1 = (2k+1)^2$.
        \item $C_{1+m\frac{(i-1)i}{2} + ij} = i$ for every $i\in\{ 1,\ldots , k \}$ and every $j\in\{ 0,\ldots , m-1 \}$.
    \end{enumerate}
\end{definition}

\begin{remark}\label{r:bien_def}
    It is straightforward to verify that the three conditions of Definition \ref{d:inicial} hold. 
    
    In addition, if we consider the function $s(i,j) = 1 + m\frac{(i-1)i}{2} + ij$ defined on the set of pairs $(i,j)$ satisfying $ (1,0) \leq _{lex} (i,j) \leq_{lex} (k,m-1) $, where $\leq _{lex}$ denotes the lexicographic order. Then $s$ is strictly increasing with respect to $\leq _{lex}$. 
    
    To see this, suppose that $(i,j) <_{lex} (u,v)$ and we distinguish between two cases:
    
    \begin{enumerate}
        \item If $i = u$ and $j < v$, then
        \[
            s(u,v) - s(i,j) = i(v-j) > 0
        \]
        
        \item If $i < u$, then
        \[ 
            \begin{aligned}
                s(i,j) & = 1 + m\frac{(i-1)i}{2} + ij \leq 1 + m\frac{(i-1)i}{2} +i(m-1) = 1 + m\frac{i(i+1)}{2} - i\\
                & < 1 + m\frac{i(i+1)}{2} < 1 + m\frac{i(i+1)}{2} + (i+1)v = s(i+1,v)
            \end{aligned} 
        \]
    \end{enumerate}

    Therefore, all subscripts are pairwise distinct, and furthermore,
    
    \begin{enumerate}
        \item $s(1,0) = 1$.
        \item $s(k,m-1) = 1 + m \frac{(k-1)k}{2} + k(m-1) = mn$.
    \end{enumerate}
\end{remark}

\begin{example}\label{e:maximal_triangular}
    To facilitate the presentation of examples of maximal triangular states, we describe a state in terms of successive blocks $\mathcal{C} = (B_1,\ldots , B_k)$. For example,
    \begin{enumerate}
        \item[$B_1:$] $(1, 1,$
        \item[$B_2:$] $2, 5, 2, 5,$
        \item[$B_3:$] $3, 6, 7, 3, 6, 7,$
        \item[$B_4:$] $4, 8, 9, 10, 4, 8, 9, 10)$
    \end{enumerate}
    It is a maximal triangular state for $m = 2$, $n = 10$ and $k = 4$. In this case, $T(\mathcal{C}) = mk = 8$.
\end{example}

\begin{lemma}\label{l:valores} 
    Let $m,n \in \mathbb{N}$, with $n$ triangular, and let $\hat{\mathcal{C}}_n ^m = (C_i)_{i=1} ^L $ be a maximal triangular state. Then $C_{X_{(i-1)m+j}} = i$ for every $j\in\{ 0,\ldots,m-1 \}$ and every $i\in\{ 1,\ldots,k \}$ with $k = \frac{1}{2}m(\sqrt{8n+1}-1)$.
\end{lemma}

\begin{proof}
    The proof is by nested strong induction. For the outer induction, we proceed by strong induction on $i$, starting from the base case $i=1$. For each fixed $i$, an inner strong induction on $j$ is carried out.
    
    Let us fix $i = 1$. For the base case $j = 0$, we obtain $C_{X_{(i-1)m + j}} = C_{X_0} = 1$.
    
    Let $j\in\{1,\ldots , m-1\}$ and assume that $C_{X_{v}} = 1 $ for every $v\in\{0,\ldots, j-1\}$. Using Lemma \ref{l:sucesion} and the strong induction hypothesis, we obtain
    \[
        X_j = 1 + \sum _{v=0} ^{j-1} C_{X_v} = 1 + \sum _{v = 0} ^{j-1} 1 = j+1
    \]
    Therefore $C_{X_j} = C_{j+1} $ and by Definition \ref{d:maximal_triangular}, we obtain $C_{X_j} = 1$.
    
    Let $i\in\{ 2,\ldots , k \}$ and $j\in\{1,\ldots , m-1\}$. Assume that $C_{X_{(u-1)m + j}} = u $ for every $u\in\{1,\ldots , i-1\}$. For the base case $j = 0$, by Lemma \ref{l:sucesion} we obtain
    \[
        X_{(i-1)m} = 1 + \sum_{r = 0} ^{(i-1)m-1} C_{X_r} = 1 + \sum_{u = 1} ^{i-1} \sum _{j=0} ^{m-1} C_{X_{(u-1)m + j}} 
    \]
    
    By the previous induction hypothesis, we obtain 
    \[
        X_{(i-1)m} = 1 + \sum_{u = 1} ^{i-1} \sum _{j=0} ^{m-1} u = 1 + m \frac{(i-1)i}{2} 
    \] 
    
    Thus, \[ C_{X_{(i-1)m + 0}} = C_{1 + m \frac{(i-1)i}{2} +0\cdot i} = i \] 
    
    For the final case, let us assume that $C_{X_{(i-1)m + v}} = i $ for every $v\in\{0,\ldots, j-1\}$, then using Lemma \ref{l:sucesion} and the previous induction hypothesis, we obtain
    \[ 
        \begin{aligned}
            X_{(i-1)m+j} & = 1 + \sum_{r = 0} ^{(i-1)m+j-1} C_{X_r} = 1 + \sum_{r = 0} ^{(i-1)m-1} C_{X_r} + \sum_{r = (i-1)m} ^{(i-1)m+j-1} C_{X_r} \\
            & = 1 + m \frac{(i-1)i}{2} + \sum_{r = (i-1)m} ^{(i-1)m+j-1} i = 1 + m \frac{(i-1)i}{2} + ij
        \end{aligned}
    \] 
    
    Finally, we conclude that \[C_{X_{(i-1)m + j}} = C_{1 + m \frac{(i-1)i}{2} + ij} = i \]
\end{proof}

\begin{proposition}\label{p:cota}
    For every maximal triangular initial state $\hat{\mathcal{C}}_n^{m} $,
    \[
        T(\hat{\mathcal{C}}_n^{m}) = \frac{1}{2} m (\sqrt{8n+1} -1)
    \]
\end{proposition}

\begin{proof}
    Let $m,n,k\in\mathbb{N}$ such that $8n+1=(2k+1)^2$ and $L = mn$. By Lemma \ref{l:sucesion} and Lemma \ref{l:valores}, we obtain
    \[
        \begin{aligned}
            X_{mk} & = 1 + \sum_{j=0} ^{mk-1} C_{X_j} = 1 + \sum _{i = 1} ^k \sum _{j = 0} ^{m-1} C_{X_{(i-1)m+j}} \\ 
            & = 1 + \sum _{i = 1} ^k \sum _{j = 0} ^{m-1} i = 1 + m\frac{k(k+1)}{2} = 1 + mn
        \end{aligned}
    \]
    
    Since $X_{mk-1} = 1 + mn - k< mn$, then $T(\hat{\mathcal{C}}_{n}^{m}) = mk = \frac{1}{2} m (\sqrt{8n+1} -1) $.
\end{proof} 

\begin{definition}\label{d:maximal_no} 
    Let $m\in\mathbb{N}$ and let $n\in\mathbb{N}$ be a not triangular number. An initial state $\hat{\mathcal{C}}_n ^m = (C_i) _{i=1} ^L$ is said to be maximal not triangular whenever the following conditions hold: \begin{enumerate}
    \item $k = \left\lfloor\frac{\sqrt{8n+1}-1}{2}\right\rfloor$, $A_k = \frac{k(k+1)}{2}$, $s = n-A_k$, $ m' = \left\lfloor\frac{m(s-1)}{k}\right\rfloor$, $a\in\{ 0,\ldots, k-1 \}$ such that $ m(s-1)  = m'k + a$ and $L = mn$.
    \item $C_{1+m\frac{(i-1)i}{2} + ij} = i$ for every $i\in\{ 1,\ldots , k \}$ and every $j\in\{ 0,\ldots , m-1 \}$.
    \item $C_{ mA_k +r(k+1) + 1} = k+1 $ for every $ r \in \{0,\ldots , m'\} $.
    \item $C_{ mA_k +m'(k+1) + a + t + 2} = k+1 $ for every integer $t$ such that $ 0 \leq t \leq m - m' - 2$.
\end{enumerate}

    An initial state is said to be maximal if it is maximal triangular or not triangular.
\end{definition}
\begin{remark} 
    When $s = 1$, we have $m' = 0$. Furthermore, if $m=1$, then $0\leq t \leq ,m-m'-2$ becomes $0 \leq t \leq -2$. However, these conditions are vacuous and do not generate any additional elements.

    As in Remark \ref{r:bien_def}, it suffices to check that the subscripts do not overlap and that they remain within the prescribed bounds.

    \begin{enumerate}
        \item For $i \in\{1,\ldots , k\}$ and $j \in \{ 0, \ldots , m-1 \} $, as shown in Remark \ref{r:bien_def}, the subscripts $s(i,j)$ are pairwise distinct. In the present setting, 
        \[ 
            s(k,m-1) = mA_k
        \]
        \item For $0\leq r \leq m'$, the sequence of subscripts $h(r) = mA_k + r(k+1) + 1$ is strictly increasing. Moreover,
        \begin{enumerate}
            \item $h(0) = mA_k + 1 > mA_k$.
            \item $h(m') = mA_k + m'(k+1) + 1$.
        \end{enumerate}
        \item For $0\leq t \leq m-m'-2$, the sequence of subscripts $g(t) = mA_k + m'(k+1) + a + t + 2$ is strictly increasing. Moreover,
        \begin{enumerate}
            \item $g(0) = mA_k + m'(k+1) + a + 2 > h(m')$.
            \item $ g(m-m'-2) = mA_k + m'(k+1) + a + m-m'-2 + 2 = mA_k + m'k + a + m = mA_k + m(s-1) + m = mA_k + ms = m(A_k + s) = mn$.
        \end{enumerate}
    \end{enumerate}
\end{remark}

\begin{example} 
    As in Example \ref{e:maximal_triangular}, we represent states by successive blocks,
    
    \[
        \mathcal{C} = (B_1,\ldots , B_k)
    \]
    
    We now consider several examples.
    
    \begin{enumerate}
    \item For $m = 2$ and $n = 7$, we have $k = 3$, and we can construct:\begin{enumerate}
        \item[$B_1:$] $(1, 1,$
        \item[$B_2:$] $2, 5, 2, 5,$
        \item[$B_3:$] $3, 6, 7, 3, 6, 7,$
        \item[$B_4:$] $4, 4)$
    \end{enumerate} 
    In this case, we obtain $T(\mathcal{C}) = mk + 1 = 7$.
    \item For $m = 5$ and $n = 7$, we have $k = 3$, and we can construct:
    \begin{enumerate}
        \item[$B_1:$] $(1, 1, 1, 1, 1$
        \item[$B_2:$] $2, 5, 2, 5, 2, 5, 2, 5, 2, 5,$
        \item[$B_3:$] $3, 6, 7, 3, 6, 7, 3, 6, 7, 3, 6, 7, 3, 6, 7,$
        \item[$B_4:$] $4, 4, 4, 4, 4)$
    \end{enumerate} 
    In this case, we obtain $T(\mathcal{C}) = mk + 2 = 17$.
    \item For $m = 3$ and $n = 9$, we have $k = 3$, and we can construct:
    \begin{enumerate}
        \item[$B_1:$] $(1, 1, 1, $
        \item[$B_2:$] $2, 5, 2, 5, 2, 5,$
        \item[$B_3:$] $3, 6, 7, 3, 6, 7, 3, 6, 7,$
        \item[$B_4:$] $4, 8, 8, 8, 4,9,9,9, 4)$
    \end{enumerate}
    In this case, we obtain $T(\mathcal{C}) = mk + 3 = 12$.
    \end{enumerate} 
\end{example}

\begin{theorem}\label{t:cota_no_triangular}
    For every maximal not triangular initial state $\hat{\mathcal{C}}_n^{m} $,
    \[
        T(\hat{\mathcal{C}}_n^{m}) = mk+m' + \left\lceil \frac{m-m'+a}{k+1} \right\rceil
    \] 
    Where $k = \left\lfloor\frac{\sqrt{8n+1}-1}{2}\right\rfloor$, $A_k = \frac{k(k+1)}{2}$, $s = n - A_k$, $m' = \left\lfloor\frac{m(s-1)}{k}\right\rfloor$ and $a\in\{0,\ldots,k-1\}$ such that $m(s-1)  = m'k + a$.
\end{theorem}
\begin{proof}
    To prove the theorem, we establish the following four claims.
    
    \medskip
    \noindent\textbf{Claim 1.} $C_{X_{mk + r}} = k+1$ for every $r\in\{ 0, \ldots , m' \}$.

    We proceed by strong induction on $r$. For the base case $r = 0$, by Lemma \ref{l:sucesion} and Lemma \ref{l:valores} (since $(C_i) _{i=1} ^{L'}$ with $L' = mA_k$, is a maximal triangular state), we have
    \[ 
        X_{mk} = 1 + \sum _{r=0} ^{mk-1} C_{X_r} = 1 + \sum_{u = 1} ^{k} \sum _{v = 0} ^{m-1} C_{X_{(u-1)m + v}}  = 1 + \sum_{u = 1} ^{k} \sum _{v = 0} ^{m-1} u = 1 + m A_k 
    \] 
    By Definition \ref{d:maximal_no}, we obtain $ C_{X_{mk}} = C_{mA_k + 1} = k+1$.

    Assume that $C_{X_{mk+p}} = k+1$ for all $p\in\{ 0, \ldots , r - 1 \}$. Then, using Lemma \ref{l:sucesion_dos} and the induction hypothesis:
    \[
        X_{mk+r} = X_{mk} + \sum _{p = mk} ^{mk + r -1} C_{X_p} = 1 + mA_k + (k+1)r
    \]
    Hence, Condition 3 of the Definition \ref{d:maximal_no} implies that $C_{X_{mA_k + r(k+1)+1}} = k+1$.

    \medskip
    \noindent\textbf{Claim 2.}  If $  k+1 < m-m'+a $ and $t\in\left\lbrace 1,\ldots , \left\lceil \frac{m-m'+a}{k+1} \right\rceil -1 \right\rbrace$, then
    \[
        C_{X_{mk + m' + t}} = k+1 \quad \text{and} \quad X_{mk+m' + \left\lceil \frac{m-m'+a}{k+1} \right\rceil} > mn 
    \] 
    
    We proceed by strong induction on $t$. For the base case $t = 1$, applying the Definition \ref{d:sucesion} and Claim 1, we obtain 
    \[ 
        \begin{aligned}
            X_{mk+m'+1} & = X_{mk+m'} + C_{X_{mk+m'}} = 1+mA_k + (k+1)m' + (k+1) \\
            & = 1+ mA_k + (m'+1)(k+1)
        \end{aligned} 
    \]
    Since $k-a \leq m-m'-2$, Condition 4 of Definition \ref{d:maximal_no} yields 
    \[
        C_{X_{mk+m'+1}} = C_{mA_k+m'(k+1) + k+2} = C_{mA_k+m'(k+1) + a + (k-a) +2} = k+1 
    \] 
    
    Assume that $C_{X_{mk+m' + p}} = k+1 $ for all $p\in\{ 1,\ldots , t-1 \}$. Then, using Lemma \ref{l:sucesion_dos} and the induction hypothesis:
    \[ 
        \begin{aligned} 
            X_{mk + m' + t} & = X_{mk+m'} + \sum_{p = mk+m'} ^{mk+m'+t-1} C_{X_p} = 1 + mA_k + (k+1)m' + (k+1)t \\ 
            & = mA_k + m'(k+1) + a + ((k + 1)t - a - 1) + 2
        \end{aligned} 
    \] 

    Since $f(t) = (k+1)t -a -1$ is strictly increasing, it suffices to evaluate it at endpoints:
    \begin{enumerate}
        \item $f(1) = k-a\geq 0$.
        \item $f\left( \left\lceil \frac{m-m'+a}{k+1} \right\rceil-1 \right) =
         (k+1)\left\lceil \frac{m-m'+a}{k+1} \right\rceil - (k+1) -a -1
        \leq (m-m'+a) + (k+1) - 1 - (k+1) - a - 1 = m-m' - 2$
    \end{enumerate}  
    Therefore, if $ 1 \leq t \leq  \left\lceil \frac{m-m'+a}{k+1} \right\rceil - 1 $, then $0\leq (k + 1)t - a - 1 \leq m-m'-2$ and consequently, by Condition 4 of Definition \ref{d:maximal_no}:
    \[ 
        C_{X_{mk + m' + t}} = k+1
    \] 
    
    Furthermore, using Definition \ref{d:sucesion} together with the above expressions, we obtain 
    \[ 
        \begin{aligned}
            X_{mk + m' + \left\lceil \frac{m-m'+a}{k+1} \right\rceil} & = X_{mk + m' + \left\lceil \frac{m-m'+a}{k+1} \right\rceil - 1} + C_{X_{mk + m' + \left\lceil \frac{m-m'+a}{k+1} \right\rceil-1}} \\
            & = 1+ mA_k + \left( m' + \left\lceil \frac{m-m'+a}{k+1} \right\rceil\right) (k+1)\\
            & \geq  1 + mA_k + m' (k+1) + m-m'+a \\
            & = 1 + mA_k + m'k + m + a\\
            & = 1 + mA_k + m'k + m + m(s-1) - m'k \\
            & = 1 + mA_k +ms = 1 + mn > mn
        \end{aligned} 
    \]

    \medskip
    \noindent\textbf{Claim 3.} If $k+1 \geq m-m'+a$, then $X_{mk+m' + 1} > mn$. 
    
    Using Definition \ref{d:sucesion} and the expressions for $X_{mk+r}$ and $C_{X_{mk+r}}$ established in Claim 1, we obtain
    \[ 
        \begin{aligned}
            X_{mk+m'+1} & = X_{mk+m'} + C_{X_{mk+m'}} =  1 + mA_k + (k+1)m' + (k+1) \\
            & = 1 + mA_k + (m'+1)(k+1) = mA_k + m'k + m' + k + 2 \\
            & = mA_k + m(s-1) + m' + k + 2 = mA_k + ms - m + m' + k + 2 \\
            & = mA_k + ms + (k+1) - (m-m'+a) + a + 1 \geq mA_k + ms + 1 + a \\
            & > mA_k + ms = mn
        \end{aligned} 
    \]

    \medskip
    \noindent\textbf{Claim 4.}  $X_{mk + m' + \left\lceil \frac{m-m'+a}{k+1} \right\rceil - 1} \leq mn$.
    
    Using the expression derived in Claim 2, we obtain
    \[ 
        \begin{aligned}
            X_{mk + m' + \left\lceil \frac{m-m'+a}{k+1} \right\rceil - 1} & = 1 + mA_k + (k+1)\left(m'+\left\lceil \frac{m-m'+a}{k+1} \right\rceil-1\right) \\
            & < 1 + mA_k + (k+1)m' + (m-m'+a) \\
            & = 1 + mA_k + m(s-1) + m = mn+1
        \end{aligned} 
    \] 
    
    If $\left\lceil \frac{m-m'+a}{k+1} \right\rceil = 1$, then the same formula is obtained from Claim 1 with $r=m'$. Hence the previous computation remains valid in this boundary case. Since $X_t$ is an integer, \[ X_{mk + m' + \left\lceil \frac{m-m'+a}{k+1} \right\rceil - 1} \leq mn \]

    Combining Claims 2 and 3, we obtain
    \[
        X_{mk + m' + \left\lceil \frac{m-m'+a}{k+1} \right\rceil} > mn 
    \]
    
    Furthermore, Claim 4 together with Definition \ref{d:sucesion} yields
    \[ 
        T(\hat{\mathcal{C}}_n^{m}) = mk + m' + \left\lceil \frac{m-m'+a}{k+1} \right\rceil
    \]
\end{proof}

In order to simplify the expression appearing in Proposition \ref{p:cota} and Theorem \ref{t:cota_no_triangular}, we introduce the following notation.

\begin{definition}\label{d:T*}
    Given $m,n\in\mathbb{N}$, let $k = \left\lfloor\frac{\sqrt{8n+1}-1}{2}\right\rfloor$, $A_k = \frac{k(k+1)}{2}$, $s = n - A_k$, $m' = \left\lfloor\frac{m(s-1)}{k}\right\rfloor$ and $a\in\{0,\ldots,k-1\}$ such that $m(s-1)  = m'k + a$. Define
    \[ 
        \sigma (m,n) =
        \begin{cases}
            mk & \text{if } s = 0 \\
            mk+m' + \left\lceil \frac{m-m'+a}{k+1} \right\rceil  & \text{if } s > 0 \\
        \end{cases} 
    \]
\end{definition}

\begin{corollary}\label{c:cota}
    For every maximal initial state $\hat{\mathcal{C}}_n^{m} $,
    \[
        T(\hat{\mathcal{C}}_n^{m}) = \sigma(m,n)
    \]
\end{corollary}

\begin{example}
    Following the idea introduced in Example \ref{e:maximal_triangular} of describing initial states by successive blocks $\mathcal{C} = (B_1,\ldots , B_k)$. We consider the following examples. 
    \begin{enumerate}
        \item For $m = 2$ and $n = 3$, we have $k = 2$, and we can construct:
        \begin{enumerate}
            \item[$B_1:$] $(1, 1,$
            \item[$B_2:$] $2, 3, 2, 3)$
        \end{enumerate}
        In this case, we obtain $T(\mathcal{C}) = 4$.
        \item Nevertheless, there exist other configurations that achieve the same $T$. For example,
        \begin{enumerate}
            \item[$B_1:$] $(1,1$
            \item[$B_2:$] $2,3,3,2)$
        \end{enumerate} 
        For this configuration, we again obtain $T(\mathcal{C}) = 4$.
    \end{enumerate}
\end{example}

\begin{theorem}\label{t:cotas} 
    For every initial state $\mathcal{C}_n^m$, 
    \[
        m \leq T({\mathcal{C} _n^m}) \leq \sigma(m,n)
    \] 
\end{theorem}

\begin{proof}
    Using Lemma \ref{l:sucesion}, we have $X_{t} = 1 + \sum _{j=0} ^{t-1} C_{X_j}$. Since $C_{X_j} \in \{1,\ldots , n\}$ for every $j$, there exist integers $N_1, \ldots , N_n \in\{0,\ldots , m\}$ such that
    \[ 
        X_t = 1 + \sum _{i=1} ^n iN_i \qquad \text{and} \qquad t = \sum_{i=1}^n N_i
    \]
    
    Where $N_i$ denotes the number of occurrences of the \emph{value} $i$ among $C_{X_0},\ldots,C_{X_{t-1}}$. In particular, for $T = T(\mathcal{C} _n^m)$, the following additional constraints hold
    \[ 
        \begin{cases}
            X_T = 1 + \sum _{i=1} ^n iN_i > mn  \\
            X_{T-1} = 1 + \sum _{i=1} ^n iN_i -a\leq mn \text{ for some } a\in\{1,\ldots , n\} \text{ such that } N_a > 0
        \end{cases} 
    \]
    
    Therefore, the problem can be formulated as the following optimization problem: 
    \[ 
        \begin{array}{rl}
            \text{opt} & T = \sum_{i=1} ^n N_i   \\
            \text{s.a.} & N_i \in \{0,\ldots , m\} \\
            & \sum _{i = 1} ^n N_i i \geq mn \\
            & \sum _{i = 1} ^n N_i i - a < mn, \text{ for some } a\in\{1,\ldots , n\} \text{ such that } N_a > 0
        \end{array} 
    \]

    For the lower bound, it suffices to observe that $Tn = \sum_{i=1} ^n N_in \geq  \sum _{i = 1} ^n N_i i \geq mn$. Therefore, $T \geq m$.

    For the upper bound, let $ \hat{T} = \sum_{i=1} ^n \hat{N}_i$ be an optimal solution, and let $\hat{a}\in\{1,\ldots , n\} $ be such that $\hat{N}_{\hat{a}} > 0$ and $\sum _{i = 1} ^n \hat{N}_i i - \hat{a} < mn$.
    
    If $l = \max\{ i: \hat{N}_i > 0 \}$ and $\hat{N}_{i-1} < m$, consider the configuration
    \[ 
        \begin{cases}
            N_l^* = \hat{N}_l - 1 \\
            N_{l-1}^* = \hat{N}_{l-1} + 1 \\
            N_i^* = \hat{N}_i \text{ for all } i \notin \{ l,l-1\}
        \end{cases} 
    \]

    Clearly $N_i ^* \in\{0,\ldots , m\}$ for every $i\in\{1,\ldots,n\}$. Moreover, if we consider $R = \sum _{i = 1} ^n \hat{N}_i i - mn$, then $R \in \{ 0,\ldots , a-1 \}$. Furthermore,
    \[
        \sum _{i = 1} ^n N_i ^* i = mn + R - 1
    \] 

    If $R \geq 1$, then $R - 1 \geq 0$ and consequently, 
    \[ 
        \sum _{i = 1} ^n N_i ^* i \geq mn
    \]
    
    Since $l = \max\{ i: \hat{N}_i > 0 \}$, we have $R <a \leq l$ and therefore $R-l < 0$. Consequently,
    \[
        \sum _{i = 1} ^n N_i ^* i - (l-1) = mn + R - l  < mn
    \] 

    Since $N^*_{l-1} > 0$, all the constraints remain satisfied when $a^* = l-1$. Consequently, we obtain another optimal solution.

    For $R = 0$, suppose that there exist $p < l$ such that $N_p ^* < m$. We construct another configuration by setting
    \[ 
        \begin{cases}
            N_p' = N_p^* + 1 \\
            N_i ' = N_i^* \text{ for all } i \neq p
        \end{cases} 
    \]

    Clearly $N_i ' \in\{0,\ldots , m\}$ for every $i\in\{1,\ldots,n\}$. Moreover,
    \[
        \sum _{i = 1} ^n N_i ' i = \sum _{i = 1} ^n N_i ^* i + p = mn - 1 + p \geq mn 
    \] 
    
    Furthermore
    \[
        \sum _{i = 1} ^n N_i ' i - p = mn - 1 < mn 
    \]
    
    Therefore, all the constraints are satisfied. However, $\sum _{i = 1} ^n N_i '  > \sum _{i = 1} ^n N_i ^*$, which contradicts the optimality of the solution.

    Therefore, whenever $R\geq 1$, the previous procedure generates a new feasible solution with the same objective value and with $R' = R-1$. Repeating this process a finite number of times, we eventually obtain a feasible solution with $R=0$.

    In this case, as shown above, the solution can be written in the following form
    \[
        \begin{cases} 
            N_i=m, & i<K,\\ 
            N_K=q,\\ 
            N_j=0, & j>K
        \end{cases}
    \] 
    
    For some $K\in\{1,\ldots,n\}$ and $q\in\{1,\ldots,m\}$.
    
    \medskip
    \noindent\textbf{Claim.} Any feasible solution of the form described above is optimal.

    If a solution of this form is feasible, then 
    \[ 
        mn \leq \sum _{i=1} ^{K-1} N_i i + qK \leq mn + a, \text{ for some }  a \in \{ 1,\ldots , K \}
    \] 
    
    If it were not optimal, then there would exist another optimal solution $T = \sum _{i = 1} ^n \tilde{N}_i$ which could be expressed as
    \[
        \begin{cases}
            \tilde{N}_i=m, & i<\tilde{K},\\
            \tilde{N}_{\tilde{K}} = \tilde{q},\\
            \tilde{N}_j=0, & j>\tilde{K}
        \end{cases}
    \]
    
    For some $\tilde{K}\in\{1,\ldots,n\}$ and $\tilde{q}\in\{1,\ldots,m\}$.
    
    In fact, we must have $\tilde{K} \geq K$. Otherwise, if $\tilde{K} < K$, then
    \[ 
        \tilde{T} = m(\tilde{K}-1) + \tilde{q}\tilde{K} \leq m\tilde{K} \leq m (K-1) \leq m(K-1) + qK = T
    \] 
    
    Which contradicts the optimality of the solution.
    
    We consider two cases. First, suppose that $\tilde{K} = K$. Then $\tilde{q} > q$ and for every $\tilde{a}$ with $\tilde{N}_{\tilde{a}} > 0$ we have $\tilde{a} \leq K$. Therefore,
    \[ 
        \sum _{i=1} ^n \tilde{N}_i i - \tilde{a} = \sum _{i=1} ^{K-1} N_i i + qK + (\tilde{q}-q)K - \tilde{a} \geq  \sum _{i=1} ^{K-1} N_i i + qK \geq mn
    \]

    Which contradicts the feasibility of the solution.

    In the case $\tilde{K} > K$, let $\tilde{a}$ be such that $\tilde{N}_{\tilde{a}} > 0$. Then $\tilde{a} \leq \tilde{K}$, and therefore,
    \[ 
        \begin{aligned}
            \sum _{i=1} ^n \tilde{N}_i i - \tilde{a} & = \sum _{i=1} ^{K-1} N_i i + qK + (m-q)K + \sum _{i = K+1} ^{\tilde{K}} \tilde{N}_i i - \tilde{a} \\
            & = \sum _{i=1} ^{K-1} N_i i + qK + (m-q)K + \tilde{q}\tilde{K} - \tilde{a} \\
            & \geq \sum _{i=1} ^{K-1} N_i i + qK + \tilde{K} - \tilde{a} \\
            & \geq \sum _{i=1} ^{K-1} N_i i + qK \geq mn
        \end{aligned}
    \]

    Which contradicts the feasibility of the solution.

    For $T = \sigma(m,n)$, Proposition \ref{p:cota} and Theorem \ref{t:cota_no_triangular} yield a feasible solution of the form described above. Therefore, by the claim, the solution is optimal.
\end{proof}

\newpage
\bibliographystyle{plain}
\bibliography{references} 

@article{Lagarias2001,
  author = {Lagarias, Jeffrey C. and Rains, Eric M. and Vanderbei, Robert J.},
  title = {The Kruskal Count},
  journal = {arXiv Mathematics e-Prints},
  year = {2001},
  eprint = {math/0110143},
  archivePrefix = {arXiv},
  primaryClass = {math.PR}
}

@incollection{Lagarias2009,
  author = {Lagarias, Jeffrey C. and Rains, Eric M. and Vanderbei, Robert J.},
  title = {The Kruskal Count},
  booktitle = {The Mathematics of Preference, Choice and Order},
  editor = {Brams, Steven J. and Gehrlein, William V. and Roberts, Fred S.},
  series = {Studies in Choice and Welfare},
  publisher = {Springer},
  address = {Berlin, Heidelberg},
  pages = {371--391},
  year = {2009},
  doi = {10.1007/978-3-540-79128-7_23}
}

@book{Lindvall2002,
  author = {Lindvall, Torgny},
  title = {Lectures on the Coupling Method},
  publisher = {Dover Publications},
  year = {2002},
  address = {Mineola, NY}
}

@article{Pollard1978,
  author = {Pollard, John M.},
  title = {Monte Carlo Methods for Index Computation (mod p)},
  journal = {Mathematics of Computation},
  volume = {32},
  number = {143},
  pages = {918--924},
  year = {1978},
  doi = {10.2307/2006496}
}

@book{Ross1996,
  author    = {Sheldon M. Ross},
  title     = {Stochastic Processes},
  publisher = {Wiley},
  year      = {1996}
}

@book{Cinlar2013,
  author    = {Erhan Cinlar},
  title     = {Probability and Stochastics},
  publisher = {Springer},
  year      = {2013}
}

@book{Durrett2019,
  author    = {Rick Durrett},
  title     = {Probability: Theory and Examples},
  publisher = {Cambridge University Press},
  year      = {2019}
}

@book{elaydi2005difference,
  author    = {Saber N. Elaydi},
  title     = {An Introduction to Difference Equations},
  edition   = {3},
  year      = {2005},
  publisher = {Springer},
  address   = {New York},
  series    = {Undergraduate Texts in Mathematics},
  isbn      = {978-0387230597},
  doi       = {10.1007/b138539}
}

@book{kelley2001difference,
  author    = {Walter G. Kelley and Allan C. Peterson},
  title     = {Difference Equations: An Introduction with Applications},
  edition   = {2},
  publisher = {Academic Press},
  address   = {San Diego, CA},
  year      = {2001},
  isbn      = {978-0124045705}
}

@article{Dobrushin1967,
  author  = {Dobrushin, R. L.},
  title   = {Shannon's Theorems for Channels with Synchronization Errors},
  journal = {Problems of Information Transmission},
  volume  = {3},
  number  = {4},
  pages   = {11--26},
  year    = {1967}
}

@book{CoverThomas2006,
  author    = {Cover, Thomas M. and Thomas, Joy A.},
  title     = {Elements of Information Theory},
  publisher = {Wiley},
  edition   = {2},
  year      = {2006}
}

@article{IyengarSiegelWolf2011,
  author  = {Iyengar, A. R. and Siegel, P. H. and Wolf, J. K.},
  title   = {On Codes for Synchronization Errors},
  journal = {IEEE Transactions on Information Theory},
  volume  = {57},
  number  = {9},
  pages   = {5667--5679},
  year    = {2011}
}

@book{BerstelPerrinReutenauer2010,
  author    = {Berstel, Jean and Perrin, Dominique and Reutenauer, Christophe},
  title     = {Codes and Automata},
  publisher = {Cambridge University Press},
  year      = {2010}
}

@article{GilbertMoore1959,
  author  = {Gilbert, Edgar N. and Moore, Edward F.},
  title   = {Variable-Length Binary Encodings},
  journal = {Bell System Technical Journal},
  volume  = {38},
  number  = {4},
  pages   = {933--967},
  year    = {1959}
}

@article{Titchener1997,
  author  = {Titchener, Michael R.},
  title   = {The Synchronization of Variable-Length Codes},
  journal = {IEEE Transactions on Information Theory},
  volume  = {43},
  number  = {2},
  pages   = {513--522},
  year    = {1997}
}

@inproceedings{Higgie1993,
  author={Higgie, G.R.},
  booktitle={Proceedings. IEEE International Symposium on Information Theory}, 
  title={Self Synchronizing T-Codes to Replace Huffman Codes}, 
  year={1993},
  volume={},
  number={},
  pages={336-336},
  doi={10.1109/ISIT.1993.748652}}

@article{Banerjee2023,
  author={Banerjee, Anisha and Lenz, Andreas and Wachter-Zeh, Antonia},
  journal={IEEE Transactions on Communications}, 
  title={Sequential Decoding of Multiple Sequences for Synchronization Errors}, 
  year={2024},
  volume={72},
  number={11},
  pages={6660-6676},
  doi={10.1109/TCOMM.2024.3405322}}

@article{Trahtman2009,
  author  = {Trahtman, Avraham N.},
  title   = {The Road Coloring Problem},
  journal = {Israel Journal of Mathematics},
  volume  = {172},
  pages   = {51--60},
  year    = {2009}
}

@inproceedings{Trahtman2008,
  author    = {Trahtman, Avraham N.},
  title     = {Synchronizing Road Coloring},
  booktitle = {Theory and Applications of Models of Computation},
  series    = {Lecture Notes in Computer Science},
  year      = {2008}
}

@inproceedings{Dedic2007,
  author    = {Dedi{\'c}, Nenad and Jakubowski, Mariusz H. and Venkatesan, Ramarathnam},
  title     = {A Graph Game Model for Software Tamper Protection},
  booktitle = {Information Hiding},
  year      = {2007}
}

@inproceedings{CharlesGorenLauter2008,
  author    = {Charles, Denis X. and Goren, Eyal Z. and Lauter, Kristin E.},
  title     = {Cryptographic Hash Functions from Expander Graphs},
  booktitle = {Journal of Cryptology},
  volume     = {22},
  number     = {1},
  pages      = {93--113},
  year        = {2009}
}

@article{JaoMillerVenkatesan2009,
  author  = {Jao, David and Miller, Stephen D. and Venkatesan, Ramarathnam},
  title   = {Expander Graphs Based on GRH with an Application to Elliptic Curve Cryptography},
  journal = {Journal of Number Theory},
  volume  = {129},
  number  = {6},
  pages   = {1491--1504},
  year    = {2009}
}

\end{document}